\documentclass[a4paper,10pt]{amsart}

\usepackage[cp850]{inputenc}
\usepackage{latexsym}
\usepackage{amsfonts}
\usepackage{amsmath}
\usepackage{amssymb,amsxtra,amscd}
\usepackage{mathrsfs}

\newcommand{\R}{\mathbb{R}}
\newcommand{\C}{\mathbb{C}}

\newcommand{\N}{\mathbb{N}}

\newcommand{\Z}{\mathbb{Z}}
\renewcommand{\Re}{\mbox{Re}\,}
\renewcommand{\Im}{\mbox{Im}\,}

\newtheorem{theo}{Theorem}
\newtheorem{lem}[theo]{Lemma}

\newtheorem{rem}[theo]{Remark}

\newtheorem{example}[theo]{Example}

\title[A characterization of chaos for weighted composition $C_0$-semigroups]{A simple characterization of chaos for weighted composition $C_0$-semigroups on Lebesgue and Sobolev spaces}
\author{T.\ Kalmes}

\begin{document}

\begin{abstract}
We give a simple characterization of chaos for weighted composition $C_0$-semigroups on $L^p_\rho(\Omega)$ for an open interval $\Omega\subseteq\R$. Moreover, we characterize chaos for these classes of $C_0$-semigroups on the closed subspace $W^{1,p}_*(\Omega)$ of the Sobolev space $W^{1,p}(\Omega)$ for a bounded interval $\Omega\subset\R$. These characterizations simplify the characterization of chaos obtained in \cite{ArKaMa13} for these classes of $C_0$-semigroups.
\end{abstract}

\maketitle

\section{Introduction}

The purpose of this article is to give a simple characterization of chaos for certain weighted composition $C_0$-semigroups on Lebesgue spaces and Sobolev spaces over open intervals. Recall that a $C_0$-semigroup $T$
on a separable Banach space $X$ is called {\it chaotic} if $T$ is hypercyclic, i.e.\ there is $x\in X$ such that $\{T(t)x;\,t\geq 0\}$ is dense in $X$,
and if the set of periodic points, i.e.\ $\{x\in X;\,\exists t>0: T(t)x=x\}$, is dense in $X$.

The study of chaotic $C_0$-semigroups has attracted the attention of many researchers. We refer the reader to Chapter 7 of the  monograph by Grosse-Erdmann and Peris \cite{GEPe11} and the references therein.
Some recent papers in the topic are
\cite{AlBaMaPe13,ArPe12,BaMo11,EmGoGo12,Ru12}.

For $\Omega\subseteq\R$ open and a Borel measure $\mu$ on $\Omega$ admitting a strictly positive Lebesgue density $\rho$ we consider $C_0$-semigroups $T$ on
$L^p(\Omega,\mu), 1\leq p<\infty,$ of the form
\[T(t)f(x)=h_t(x) f(\varphi(t,x)),\]
where $\varphi$ is the solution semiflow of an ordinary differential equation
\[\dot{x}=F(x)\]
in $\Omega$ and
\[h_t(x)=\exp\big(\int_0^t h(\varphi(s,x))ds\big)\]
with $h\in C(\Omega)$. Such $C_0$-semigroups appear in a natural way when dealing with initial value problems for linear first order partial differential operators. While a characterization of chaos for such $C_0$-semigroups was obtained for open $\Omega\subseteq\R^d$ for arbitrary dimension $d$ in \cite{Ka07}, evaluation of
these conditions in concrete examples is sometimes rather involved. In contrast to general dimension the case $d=1$ allows for a significantly simplified characterization, see \cite{ArKaMa13}. However, this characterization of chaos still depends on the knowledge of the solution semiflow $\varphi$ which might be difficult to determine in concrete examples.

In section 2 we give, under mild additional assumptions on $F$ and $h$, a characterization of chaos which only depends on the ingredients $F$, $h$, and $\rho$, without refering to the semiflow $\varphi$.

In section 3 we use this result to obtain a similarly simple characterization of chaos for the above kind of $C_0$-semigroups acting on the closed subspace
\[W^{1,p}_*[a,b]=\{f\in W^{1,p}[a,b];\,f(a)=0\}\]
of the Sobolev spaces $W^{1,p}[a,b]$, where $(a,b)\subseteq\R$ is a bounded interval. It was shown in \cite{ArKaMa13} that such $C_0$-semigroups cannot be hypercyclic, a fortiori chaotic, on the whole Sobolev space $W^{1,p}(a,b)$.

In order to illustrate our results, several examples are considered.

\section{Chaotic weighted composition $C_0$-semigroups on Lebesgue spaces}

Let $\Omega\subseteq\R$ be open and let $F:\Omega\rightarrow\R$ be a $C^1$-function. Hence, for every $x_0\in\Omega$ there is a unique solution $\varphi(\cdot,x_0)$ of the initial value problem
\[\dot{x}=F(x),\; x(0)=x_0.\]
Denoting its maximal domain of definition by $J(x_0)$ it is well-known that $J(x_0)$ is an open interval containing $0$. We make the general assumption that $\Omega$ is {\it forward invariant under} $F$, i.e.\
$[0,\infty)\subset J(x_0)$ for every $x_0\in\Omega$, that is $\varphi:[0,\infty)\rightarrow\Omega$. This is true, for example, if $\Omega=(a,b)$ is a bounded interval and if $F$ can be extended to a $C^1$-function defined on a neighborhood of $[a,b]$ such that $F(a)\geq 0$ and $F(b)\leq 0$ (cf.\ \cite[Corollary 16.10]{Amann}).

From the uniqueness of the solution it follows that $\varphi(t,\cdot)$ is injective for
every $t\geq 0$ and $\varphi(t+s,x)=\varphi(t,\varphi(s,x))$ for all $x\in\Omega$ and $s,t\in J(x)$ with $s+t\in J(x)$. Moreover, for every $t\geq 0$ the set $\varphi(t,\Omega)$ is open and for $x\in\varphi(t,\Omega)$
we have $[-t,\infty)\subset J(x)$ as well as $\varphi(-s,x)=\varphi(s,\cdot)^{-1}(x)$ for all $s\in [0,t]$. Since $F$ is a $C^1$-function it is well-known that the same is true for
$\varphi(t,\cdot)$ on $\Omega$ and $\varphi(-t,\cdot)$ on $\varphi(t,\Omega)$ for every $t\geq 0$.

Moreover, let $h\in C(\Omega)$ and define for $t\ge 0$
\[h_t:\Omega\rightarrow\C, h_t(x)=\exp(\int_0^t h(\varphi(s,x))ds).\]
For $1\leq p<\infty$ and a measurable function $\rho:\Omega\rightarrow (0,\infty)$ let $L^p_\rho(\Omega)$ be as usual
the Lebesgue space of $p$-integrable functions with respect to the Borel measure $\rho d\lambda$, where $\lambda$
denotes Lebesgue measure. If $\Omega$ is forward invariant under $F$ the operators
\[T(t):L^p_\rho(\Omega)\rightarrow L^p_\rho(\Omega), (T(t)f)(x):=h_t(x)f(\varphi(t,x))\;(t\geq 0)\]
are well-defined continuous linear operators defining a $C_0$-semigroup $T_{F,h}$ on $L^p_\rho(\Omega)$ if $\rho$ is
$p$-{\it admissible for} $F$ {\it and} $h$, i.e.\ if there are constants $M\geq 1$, $\omega\in\R$ with
\[\forall\,t\geq 0,x\in\Omega:\,|h_t(x)|^p\rho(x)\leq M e^{\omega t}\rho(\varphi(t,x))\exp(\int_0^t F'(\varphi(s,x))ds),\]
(see \cite{ArKaMa13}). Because $|h_t(x)|^p=\exp(p\int_0^t\Re h(\varphi(s,x))ds)$ it follows that $\rho=1$ is 
$p$-admissible for any $p$ if $\Re h$ is bounded above and $F'$ is bounded below, i.e.\ in this case the above operators
define a $C_0$-semigroup $T_{F,h}$ on the standard Lebesgue spaces $L^p(\Omega)$. Under mild additional assumptions on $F$ and $h$ the generator of 
this $C_0$-semigroup is given by the first order differential operator $Af=F f'+hf$ on a suitable subspace of 
$L^p(\Omega)$ (see \cite[Theorem 15]{ArKaMa13}).

In \cite[Theorem 6 and Proposition 9]{ArKaMa13} it is characterized when the $C_0$-semigroup $T_{F,h}$ is chaotic on $L^p_\rho(\Omega)$. However, this characterization depends on a more or less
explicit knowledge of the semiflow $\varphi$.

Our aim is to prove the following characterization of chaos for $T_{F,h}$ on $L^p_\rho(\Omega)$ valid under mild additional assumptions on $F$ and $h$ and which is given solely in terms of $F$, $h$, and $\rho$. Throughout this article, we use the following common abbreviation $\{F=0\}:=\{x\in\Omega;\,F(x)=0\}$.  

\begin{theo}\label{simplified chaos}
	For $1\leq p<\infty$ let $\Omega\subset\R$ be an open interval which is forward invariant under $F\in C^1(\Omega)$, and let $h\in C(\Omega)$ be such
	that $F'$ and $\Re h$ are bounded and
	\begin{itemize}
		\item[a)] There is $\gamma\in\R$ such that $h(x)=\gamma$ for all $x\in\{F=0\}$.
		\item[b)] With $\alpha:=\inf\Omega$ and $\omega:=\sup\Omega$ the function
		\[\Omega\rightarrow\C,y\mapsto\frac{\Im h(y)}{F(y)}\]
		belongs to $L^1((\alpha,\beta))$ for all $\beta\in\Omega$ or to $L^1((\beta,\omega))$ for all $\beta\in\Omega$.
	\end{itemize}
	Then for every $\rho$ which is $p$-admissible for $F$ and $h$ the following are equivalent.
	\begin{itemize}
		\item[i)] $T_{F,h}$ is chaotic in $L^p_\rho(\Omega)$.
		\item[ii)] $\lambda(\{F=0\})=0$ and for every connected component $C$ of $\Omega\backslash\{F=0\}$
		\[\int_C\exp(-p\int_x^w\frac{\Re h(y)}{F(y)}dy)\rho(w)d\lambda(w)<\infty\]
		for some/all $x\in C$.
	\end{itemize}
\end{theo}

In order to prove Theorem \ref{simplified chaos}, we define for $x\in\Omega$, $p\geq 1$, and $t\geq 0$
\begin{eqnarray*}
	\rho_{t,p}(x)&=&\chi_{\varphi(t,\Omega)}(x)|h_t(\varphi(-t,x))|^p\exp(\int_0^{-t}F'(\varphi(s,x))ds)\,\rho(\varphi(-t,x))\\
	&=&\chi_{\varphi(t,\Omega)}(x)\exp(p\int_0^t\Re h(\varphi(s,\varphi(-t,x)))ds)\exp(\int_0^{-t}F'(\varphi(s,x))ds)\,\rho(\varphi(-t,x))
\end{eqnarray*}
as well as
\begin{eqnarray*}
	\rho_{-t,p}(x)&=&|h_t(x)|^{-p}\exp(\int_0^t F'(\varphi(s,x))ds)\,\rho(\varphi(t,x))\\
	&=&\exp(-p\int_0^t\Re h(\varphi(s,x))ds)\exp(\int_0^t F'(\varphi(s,x))ds)\,\rho(\varphi(t,x)).
\end{eqnarray*}
Then $\rho_{0,p}=\rho$, $\rho_{t,p}\geq 0$ for every $t\in\R$, and for fixed $x\in\Omega$ the mapping $t\mapsto\rho_{t,p}(x)$ is Lebesgue measurable. Moreover, it follows that
\begin{eqnarray}\label{Nummer1}
	\rho_{-(t+s),p}(x)&=&\exp(-p\int_0^{t+s}\Re h(\varphi(r,x))\,dr)\nonumber\\
	&&\cdot\exp(\int_0^{t+s}F'(\varphi(r,x))dr)\rho(\varphi(t,\varphi(s,x)))\nonumber\\
	&=&\exp(\int_0^s F'(\varphi(r,x))-p\Re h(\varphi(r,x))dr)\nonumber\\
	&&\cdot\exp(-p\int_0^t\Re h(\varphi(r,\varphi(s,x)))dr)\\
	&&\cdot\exp(\int_0^t F'(\varphi(r,\varphi(s,x)))dr)\rho(\varphi(t,\varphi(s,x)))\nonumber\\
	&=&\exp(\int_0^s F'(\varphi(r,x))-p\Re h(\varphi(r,x))dr)\rho_{-t,p}(\varphi(s,x))\nonumber
\end{eqnarray}
and analogously
\begin{eqnarray}\label{Nummer2}
	\rho_{(t+s),p}(x)&=&\chi_{\varphi(t+s,\Omega)}(x)\exp(p\int_{-(t+s)}^0\Re h(\varphi(r,x))-\frac{1}{p}F'(\varphi(r,x))dr)\nonumber\\
	&&\cdot\rho(\varphi(-(t+s),x))\\
	&=&\chi_{\varphi(s,\Omega)}(x)\exp(p\int_{-s}^0\Re h(\varphi(r,x))-\frac{1}{p}F'(\varphi(r,x))dr)\rho_{t,p}(\varphi(-s,x)).\nonumber
\end{eqnarray}

The following lemma will be used in the proof of the first auxilary result. We cite it for the reader's convenience.
For a proof see \cite[Lemma 7]{Ka09}.

\begin{lem}\label{aux}
	Let $\Omega\subseteq\R$ be open, let $F\in C^1(\Omega)$ be such that $\Omega$ is forward invariant under $F$, and let
	$h\in C(\Omega)$ be real valued. Moreover, for fixed $1\leq p<\infty$ let $\rho$ be $p$-admissible for $F$ and $h$. For
	$[a,b]\subset\Omega\backslash\{F=0\}$ set $\alpha:=a$ and $\beta:=b$ if $F_{|[a,b]}>0$, respectively $\alpha:=b$
	and $\beta:=a$ if $F_{|[a,b]}<0$.
	
	Then there is a constant $C>0$ such that
	\[\forall\,x\in [a,b]:\,\frac{1}{C}\leq\rho(x)\leq C\]
	as well as
	\[\forall\,t\in\R, x\in[a,b]:\,\frac{1}{C}\rho_{t,p}(\alpha)\leq\rho_{t,p}(x)\leq C\rho_{t,p}(\beta).\]
\end{lem}

\begin{lem}\label{series}
	Let $\Omega\subseteq\R$ be open and forward invariant under $F\in C^1(\Omega)$, let $h\in C(\Omega)$ be such that $F'$ and $\Re h$ are bounded. Moreover, let $\rho$ be $p$-admissible for $F$ and $h$, $1\leq p<\infty$. Then the following are equivalent.
	\begin{itemize}
		\item[i)] For all $x\in\Omega\backslash\{F=0\}$ there is $t_0>0$ such that $\sum_{k\in\Z}\rho_{k t_0,p}(x)<\infty$.
		\item[ii)] For all $x\in\Omega\backslash\{F=0\}: \int_\R \rho_{t,p}(x)d\lambda(t)<\infty$.
		\item[iii)] For all $x\in\Omega\backslash\{F=0\}$ and $t_0>0: \sum_{k\in\Z}\rho_{k t_0,p}(x)<\infty.$
	\end{itemize}
\end{lem}

\begin{proof}
	In order to show that i) implies ii) fix $x\in\Omega\backslash\{F=0\}$ and choose $t_0>0$ according to i) for $x$. We distinguish two cases. If $x$ belongs to $\cap_{t\geq 0}\varphi(t,\Omega)$ it follows by equation (\ref{Nummer2}) and the boundedness of $\Re h$ and $F'$
	\begin{eqnarray*}
		&&\int_{[0,\infty)}\rho_{t,p}(x)d\lambda(t)=\sum_{k=0}^\infty\int_{[0,t_0]}\rho_{k t_0+s,p}(x)d\lambda(s)\\
		&=&\sum_{k=0}^\infty\int_{[0,t_0]}\chi_{\varphi(s,\Omega)}(x)\exp(p\int_{-s}^0\Re h(\varphi(r,x))-\frac{1}{p}F'(\varphi(r,x))dr)\\
		&&\rho_{kt_0,p}(\varphi(-s,x))d\lambda(s)\\
		&\leq& C\sum_{k=0}^\infty\int_{[0,t_0]}\chi_{\varphi(s,\Omega)}(x)\rho_{k t_0,p}(\varphi(-s,x))d\lambda(s)\\
		&=& C\sum_{k=0}^\infty\int_{[0,t_0]}\rho_{k t_0,p}(\varphi(-s,x))d\lambda(s)\\
		&\leq &\begin{cases}\tilde{C}\sum_{k=0}^\infty\rho_{k t_0,p}(x)&, F(x)>0\\ \tilde{C}\sum_{k=0}^\infty\rho_{k t_0,p}(\varphi(-t_0,x))&, F(x)<0,\end{cases}
	\end{eqnarray*}
where $C$ and $\tilde{C}$ depend on $t_0$ and where in the last step we used lemma \ref{aux} for $F$ and $\Re h$.
Since by equation (\ref{Nummer2}) together with the boundedness of $F'$ and $\Re h$ we also have with suitable $D>0$ that for all $k\geq 0$
	\[\rho_{k t_0,p}(\varphi(-t_0,x))\leq D\rho_{(k+1)t_0,p}(x),\]
the above shows the existence of $\hat{C}>0$ such that
	\[\int_{[0,\infty)}\rho_{t,p}(x)d\lambda(t)\leq \hat{C}\sum_{k=0}^\infty\rho_{k t_0,p}(x).\]
	
If $x$ does not belong to $\cap_{t\geq 0}\varphi(t,\Omega)$ then $\int_{[0,\infty)}\rho_{t,p}(x)d\lambda(t)=\int_{[0,r]}\rho_{t,p}(x)d\lambda(t)$ for some $r>0$. Combining lemma \ref{aux} for $F$ and $\Re h$ with equation (\ref{Nummer2}), the boundedness of $F'$ and $\Re h$ gives for suitable $C>0$
	\begin{eqnarray*}
		&&\int_{[0,\infty)}\rho_{t,p}(x)d\lambda(t)\\
		&=&\int_{[0,r]}\chi_{\varphi(t,\Omega)}(x)\exp(p\int_{-t}^0\Re h(\varphi(s,x))-\frac{1}{p}F'(\varphi(s,x))ds)\rho_{0,p}(\varphi(-t,x))d\lambda(t)\\
		&\leq& C\int_{[0,r]}\chi_{\varphi(t,\Omega)}(x)\rho_{0,p}(\varphi(-t,x))d\lambda(t)<\infty.
	\end{eqnarray*}
Thus, if i) holds then $\int_{[0,\infty)}\rho_{t,p}(x)d\lambda(t)<\infty$ for all $x\in\Omega\backslash\{F=0\}$.
	
Moreover, by equation (\ref{Nummer1}) we obtain for every $x\in\Omega\backslash\{F=0\}$ together with the boundedness of $F'$ and $\Re h$
\begin{eqnarray*}
	&&\int_{(-\infty,0]}\rho_{t,p}(x)d\lambda(t)=\sum_{k=0}^\infty\int_{[-t_0,0]}\rho_{-(k t_0-s),p}(x)d\lambda(s)\\
	&=&\sum_{k=0}^\infty\int_{[-t_0,0]}\exp(\int_0^{-s} F'(\varphi(r,x))-p\Re h(\varphi(r,x)))\rho_{-k t_0,p}(\varphi(-s,x))d\lambda(s)\\
	&\leq& C\sum_{k=0}^\infty\int_{[-t_0,0]}\rho_{-k t_0,p}(\varphi(-s,x))d\lambda(s)\\
	&\leq& \begin{cases}\tilde{C}\sum_{k=0}^\infty \rho_{-k t_0,p}(\varphi(t_0,x))&, F(x)>0\\ \tilde{C}\sum_{k=0}^\infty \rho_{-k t_0,p}(x)&, F(x)<0,\end{cases}
\end{eqnarray*}	
where $C$ and $\tilde{C}$ again depend on $t_0$ and where in the last step we again used lemma \ref{aux} for $F$ and $\Re h$. Equation (\ref{Nummer1}) and the fact that $F'$ and $\Re h$ are bounded yield the existence of $D>0$ such that for all $k\geq 0$
\[\rho_{-k t_0,p}(\varphi(t_0,x))\leq D\rho_{-(k+1)t_0,p}(x).\]
So the above gives
\[\int_{(-\infty,0]}\rho_{t,p}(x)d\lambda(t)\leq\hat{C}\sum_{k=0}^\infty \rho_{-kt_0,p}(x)\]
for some $\hat{C}>0$. Hence, i) implies ii).
	
In order to show that ii) implies iii) we fix $t_0>0$ and $x\in\Omega\backslash\{F=0\}$ and distinguish again two cases. If $x$ does not belong to $\cap_{t\geq 0}\varphi(t,\Omega)$ there is $t_1>0$ such that $\rho_{t,p}(x)=0$ for all $t>t_1$. Therefore, $\sum_{k=0}^\infty\rho_{k t_0,p}(x)<\infty$.
	
In case of $x\in\cap_{t\geq 0}\varphi(t,\Omega)$ it follows from equation (\ref{Nummer2}) together with the boundedness of $F'$ and $\Re h$ that for some $C>0$
\begin{eqnarray*}
	&&\int_{[0,\infty)}\rho_{t,p}(x)d\lambda(t)=\sum_{k=0}^\infty\int_{[0,t_0]}\rho_{k t_0+t,p}(x)d\lambda(t)\\
	&=&\sum_{k=0}^\infty\int_{[0,t_0]}\exp(p\int_{-t}^0\Re h(\varphi(r,x))-\frac{1}{p}F'(\varphi(r,x))dr)\rho_{k t_0,p}(\varphi(-t,x))d\lambda(t)\\
	&\geq&\sum_{k=0}^\infty C\int_{[0,t_0]}\rho_{k t_0,p}(\varphi(-t,x))d\lambda(t)\\
	&\geq&\begin{cases}\tilde{C}\sum_{k=0}^\infty\rho_{k t_0,p}(x)&, F(x)<0\\ \tilde{C}\rho_{k t_0,p}(\varphi(-t_0,x))&, F(x)<0,\end{cases}
\end{eqnarray*}
where we used lemma \ref{aux} in the last step. By equation (\ref{Nummer2}) and the boundedness of $F'$ and $\Re h$ we have $\rho_{k t_0,p}(\varphi(-t_0,x))\geq D\rho_{(k+1)t_0,p}(x)$ for suitable $D>0$ such that the above gives
\begin{equation}\label{Nummer3}
\int_{[0,\infty)}\rho_{t,p}(x)d\lambda(t)\geq\hat{C}_1\sum_{k=0}^\infty\rho_{k t_0,p}(x)
\end{equation}
for some $\hat{C}_1$. 
	
Additionally, applying lemma \ref{aux} for $F$ and $\Re h$ we also obtain from the boundedness of $F'$ and $\Re h$ together with equation (\ref{Nummer1}) 
\begin{eqnarray*}
	&&\int_{(-\infty,0]}\rho_{t,p}(x)d\lambda(t)=\sum_{k=0}^\infty\int_{[-t_0,0]}\rho_{-(k t_0-t),p}(x)d\lambda(t)\\
	&=&\sum_{k=0}^\infty\int_{[-t_0,0]}\exp(\int_0^{-t}F'(\varphi(r,x))-p\Re h(\varphi(r,x))dr)\rho_{-k t_0,p}(\varphi(-t,x))d\lambda(t)\\
	&\geq & C\sum_{k=0}^\infty\int_{[-t_0,0]}\rho_{-k t_0,p}(\varphi(-t,x))d\lambda(t)\\
	&\geq & \begin{cases}\tilde{C}\sum_{k=0}^\infty\rho_{-k t_0,p}(x)&, F(x)>0\\ \tilde{C}\sum_{k=0}^\infty\rho_{-k t_0,p}(\varphi(-t_0,x))&, F(x)<0\end{cases}\\
	&\geq &\hat{C}_2\sum_{k=0}^\infty \rho_{-k t_0,p}(x).
\end{eqnarray*}
Hence, together with (\ref{Nummer3}), iii) follows from  ii), and as iii) obviously implies i), the lemma is proved.
\end{proof}

The applicability of the previous lemma depends on an explicit knowledge of $\varphi$. The next lemma shows that
the integrals appearing in the previous result can be expressed in terms of $F$, $h$, and $\rho$.

\begin{lem}\label{integral}
	Let $\Omega\subseteq\R$ be open and forward invariant under $F\in C^1(\Omega)$, $h\in C(\Omega)$ and let $\rho$ be
	$p$-admissible for $F$ and $h$, $1\leq p<\infty$. Then for every $x\in\Omega\backslash\{F=0\}$ we have
	\[\int_\R\rho_{t,p}(x)d\lambda(t)=\frac{1}{|F(x)|}\int_{C(x)}\exp(p\int_w^x\frac{\Re h(y)}{F(y)}dy)\rho(w)d\lambda(w),\]
	where $C(x)$ denotes the connected component of $\Omega\backslash\{F=0\}$ containing $x$.
\end{lem}

\begin{proof}
Fix $x\in\Omega\backslash\{F=0\}$ and let $C(x)$ be as in the lemma. Observe that $\varphi(t,x)\in C(x)$ for all
$t\in J(x)$ and that $\varphi(J(x),x)=C(x)$, where $J(x)$ is the domain of the maximal solution $\varphi(\cdot,x)$ of the initial value problem $\dot{y}=F(y), y(0)=x$. Obviously,
\[\int_\R \rho_{t,p}(x)d\lambda(t)=\int_{[0,\infty)}\rho_{t,p}(x)d\lambda(t)+\int_{[0,\infty)}\rho_{-t,p}(x)d\lambda(x).\]
We set $C^+(x)=\{\varphi(t,x);\,t\geq 0\}$. Applying the Transformation Formula for Lebesgue integrals we obtain with equation (\ref{Nummer1})
\begin{eqnarray*}
	&&\int_{[0,\infty)}\rho_{-t,p}(x)d\lambda(t)\\
	&=&\int_{[0,\infty)}\exp(\int_0^t F'(\varphi(r,x))-p\Re h(\varphi(r,x)) dr)\rho(\varphi(t,x))d\lambda(t)\\
	&=&\int_{[0,\infty)}\exp(\int_0^t\frac{F'(\varphi(r,x))-p\Re h(\varphi(r,x))}{F(\varphi(r,x))}\partial_1\varphi(r,x) dr)\rho(\varphi(t,x))d\lambda(t)\\
	&=&\int_{[0,\infty)}\exp(\int_x^{\varphi(t,x)}\frac{F'(y)-p\Re h(y)}{F(y)}dy)\rho(\varphi(t,x))d\lambda(t)
\end{eqnarray*}
\begin{eqnarray*}
	&=&\int_{(0,\infty)}\frac{\exp(\int_x^{\varphi(t,x)}\frac{F'(y)-p\Re h(y)}{F(y)} dy)}{|F(\varphi(t,x))|}\rho(\varphi(t,x))|\partial_1\varphi(t,x)|d\lambda(t)\\
	&=&\int_{C^+(x)}\frac{\exp(\int_x^w \frac{F'(y)-p\Re h(y)}{F(y)}dy)}{|F(w)|}\rho(w)d\lambda(w)\\
	&=&\int_{C^+(x)}\frac{\exp(\int_w^x \frac{p\Re h(y)-F'(y)}{F(y)}dy)}{|F(w)|}\rho(w)d\lambda(w).
\end{eqnarray*}
Moreover, denoting $\alpha=\sup\{t\geq 0;\,x\in\varphi(t,\Omega)\}$ we have $-\alpha=\inf J(x)$. With $C^-(x)=\varphi((-\alpha,0],x)$ it follows $C(x)=C^+(x)\cup C^-(x)$, $C^+(x)\cap C^-(x)=\{x\}$, and
\begin{eqnarray*}
	&&\int_{[0,\infty)}\rho_{t,p}(x)d\lambda(t)=\int_{[0,\alpha)}\rho_{t,p}(x)d\lambda(t)\\
	&=&\int_{[0,\alpha)}\exp(\int_{-t}^0 p\Re h(\varphi(r,x))-F'(\varphi(r,x))dr)\rho(\varphi(-t,x))d\lambda(t)\\
	&=&\int_{(-\alpha,0]}\exp(\int_t^0 p\Re h(\varphi(r,x))-F'(\varphi(r,x))dr)\rho(\varphi(t,x))d\lambda(t)\\
	&=&\int_{(-\alpha,0]}\exp(\int_t^0\frac{p\Re h(\varphi(r,x))-F'(\varphi(r,x))}{F(\varphi(r,x))}\partial_1\varphi(r,x)dr)\rho(\varphi(t,x))d\lambda(t)\\
	&=&\int_{(-\alpha,0]}\exp(\int_{\varphi(t,x)}^x\frac{p\Re h(y)-F'(y)}{F(y)}dy)\rho(\varphi(t,x))d\lambda(t)\\
	&=&\int_{(-\alpha,0]}\frac{\exp(\int_{\varphi(t,x)}^x\frac{p\Re h(y)-F'(y)}{F(y)}dy)}{|F(\varphi(t,x))|}\rho(\varphi(t,x))|\partial_1\varphi(t,x)|d\lambda(t)\\
	&=&\int_{C^-(x)}\frac{\exp(\int_w^x\frac{p\Re h(y)-F'(y)}{F(y)}dy)}{|F(w)|}\rho(w)d\lambda(w).
\end{eqnarray*}
Combining these equations yields
\begin{eqnarray*}
	&&\int_\R \rho_{t,p}(x)d\lambda(t)\\
	&=&\int_{C(x)}\frac{\exp(\int_w^x\frac{p\Re h(y)-F'(y)}{F(y)}dy)}{|F(w)|}\rho(w)d\lambda(w)\\
	&=&\int_{C(x)}\frac{\exp(p\int_w^x\frac{\Re h(y)}{F(y)}dy)\exp(\log|F(w)|-\log|F(x)|)}{|F(w)|}\rho(w)d\lambda(w)\\
	&=&\frac{1}{|F(x)|}\int_{C(x)}\exp(p\int_w^x\frac{\Re h(y)}{F(y)}dy)\rho(w)d\lambda(w)
\end{eqnarray*}
which proves the lemma. 
\end{proof}

\begin{rem}\label{connected component}
\begin{rm}
The last step in the above proof shows that for $x\in\Omega\backslash\{F=0\}$ and all $v\in C(x)$ we have for every
$1\leq p<\infty$
\begin{eqnarray*}
	&&\int_\R\rho_{t,p}(x)d\lambda(t)\\
	&=&\frac{1}{|F(x)|}\int_{C(x)}\exp(p\int_w^x\frac{\Re h(y)}{F(y)}dy)\rho(w)d\lambda(w)\\
	&=&\frac{|F(v)|}{|F(x)|}\exp(p\int_v^x\frac{\Re h(y)}{F(y)}dy)\frac{1}{|F(v)|}\int_{C(x)}\exp(p\int_x^v\frac{\Re h(y)}{F(y)}dy)\rho(w)d\lambda(w)\\
	&=&\frac{|F(v)|}{|F(x)|}\exp(p\int_v^x\frac{\Re h(y)}{F(y)}dy)\int_\R\rho_{t,p}(v)d\lambda(t).
\end{eqnarray*}
Thus, under the hypotheses of Lemma \ref{integral} the following are equivalent for every connnected component $C$ of $\Omega\backslash\{F=0\}$ and all $1\leq p<\infty$.
\begin{itemize}
	\item[i)] $\exists\,x\in C:\,\int_\R\rho_{t,p}(x)d\lambda(t)<\infty$,
	\item[ii)] $\forall\,x\in C:\,\int_\R\rho_{t,p}(x)d\lambda(t)<\infty$,
	\item[iii)] $\exists\,x\in C:\,\int_C\exp(-p\int_x^w\frac{\Re h(y)}{F(y)}dy)\rho(w)d\lambda(w)<\infty$,
	\item[iv)] $\forall\,x\in C:\,\int_C\exp(-p\int_x^w\frac{\Re h(y)}{F(y)}dy)\rho(w)d\lambda(w)<\infty$.
\end{itemize}
\end{rm}
\end{rem}

We have now everything at hand to prove Theorem \ref{simplified chaos}.\\

{\it Proof of Theorem \ref{simplified chaos}.} By \cite[Theorem 6 and Proposition 9]{ArKaMa13} $T_{F,h}$ is chaotic on $L^p_\rho(\Omega)$ if and only if $\lambda(\{F=0\})=0$ as well as for every $m\in\N$ for which there are $m$ different connected components $C_1,\ldots,C_m$ of $\Omega\backslash \{F=0\}$, for $\lambda^m$-almost all choices of $(x_1,\ldots,x_m)\in\Pi_{j=1}^m C_j$ there is $t>0$ such that
\[\sum_{j=1}^m\sum_{l\in\Z}\rho_{lt,p}(x_j)<\infty.\]
By lemma \ref{series}, this holds precisely when $\lambda(\{F=0\})=0$ and when for $\lambda$-almost every $x\in\Omega\backslash\{F=0\}$
\[\int_\R\rho_{t,p}(x)d\lambda(t)<\infty.\]
Thus, applying Remark \ref{connected component}, Theorem \ref{simplified chaos} follows.\hfill$\square$.\\

\begin{rem}\label{remark}
\begin{rm}
a) Inspection of the proof of Theorem \ref{simplified chaos} yields the following. Under the hypothesis of Theorem \ref{simplified chaos}, the following are equivalent for $\rho$ $p$-admissible for $F$ and $h$.
\begin{itemize}
	\item[i)] $T_{F,h}$ is chaotic in $L^p_\rho(\Omega)$.
	\item[ii)] $\lambda(\{F=0\})=0$ and for all $x\in\Omega\backslash\{F=0\}$ there is $t_0>0$ such that 
	$\sum_{k\in\Z}\rho_{k t_0,p}(x)<\infty$.
	\item[iii)] $\lambda(\{F=0\})=0$ and $\sum_{k\in\Z}\rho_{k t_0,p}(x)<\infty$ for all $x\in\Omega\backslash\{F=0\}$ 
	and all $t_0>0$.
	\item[iv)] $\lambda(\{F=0\})=0$ and $\int_\R\rho_{t,p}(x)d\lambda(t)<\infty$ for all $x\in\Omega\backslash\{F=0\}$.
	\item[v)] $\lambda(\{F=0\})=0$ and for every connected component $C$ of $\Omega\backslash\{F=0\}$
	\[\int_C\exp(-p\int_x^w\frac{\Re h(y)}{F(y)}dy)\rho(w)d\lambda(w)<\infty\]
	for some/all $x\in C$.
\end{itemize}

b) If $h=0$ and if $F\in C^1(\Omega)$ is as usual then the $p$-admissibility of $\rho$ does not depend on $p$. If moreover
$F'$ is bounded the following are then equivalent.
\begin{itemize}
	\item[i)] $T_F=T_{F,0}$ is chaotic in $L^p_\rho(\Omega)$ for some/all $p\in[1,\infty)$.
	\item[ii)] $\lambda(\{F=0\})=0$ and for every connected component $C$ of $\Omega\backslash\{F=0\}$ we have
	\[\int_C\rho(w)d\lambda(w)<\infty.\]
\end{itemize}
\end{rm}
\end{rem}

\begin{example}
\begin{rm}
a) Let $\Omega\in\{(0,\infty),\R\}$ and let $F(x)=1$. Then $\Omega$ is forward invariant under $F$. Moreover, let $h\in C(\Omega)$ be such that $\Re h$ is bounded. It follows from the definition, that $\rho=1$ is $p$-admissible for $F$ and $h$ for every $1\leq p<\infty$ so that $T_{1,h}$ is a well defined $C_0$-semigroup on $L^p(\Omega)$, the so-called perturbed translation semigroup. If $h$ is bounded the generator of $T_{1,h}$ in $L^p(\Omega)$ is given by
\[A_p:W^{1,p}(\Omega)\rightarrow L^p(\Omega), A_pf(x)=f'+hf,\]
where $f'$ denotes the distributional derivative of $f$ (see e.g.\ \cite[Theorem 15]{ArKaMa13}).

If $\Im h\in L^1(0,\beta)$, resp.\ $\Im h\in L^1(-\infty,\beta)$ for all $\beta\in\Omega$ or if $\Im h\in L^1(\beta,\infty)$ for all $\beta\in\Omega$, by Theorem \ref{simplified chaos} this $C_0$-semigroup is chaotic on $L^p(\Omega)$ if and only if
\[\int_\Omega\exp(-p\int_1^w\Re h(y)dy)d\lambda(w)<\infty.\]

b) Consider again $\Omega\in\{(0,\infty),\R\}$ and let $F(x)=1$. Moreover, let $\rho$ be $p$-admissible for $F$ and $h=0$ (which does not depend on $p$ by Remark \ref{remark} b)). We then obtain the classical translation semigroup and Remark \ref{remark} a) gives the well-known characterizations of chaos for this semigroup due to Matsui, Yamada, and Takeo \cite{MaYaTa03,MaYaTa04} and deLaubenfels and Emamirad \cite{deLaEm01}, respectively.

c) Consider $\Omega=(0,1)$ and let $F(x)=-x$. Then $\Omega$ is forward invariant for $F$. Additionally, let $h\in C(0,1)$ be such that $\Re h$ is bounded. It follows again from the definition that $\rho=1$ is $p$-admissible for $F$ and $h$ for every $1\leq p<\infty$. Thus, we obtain a well-defined $C_0$-semigroup $T_{-id,h}$ on $L^p(0,1)$. If $h$ is bounded the generator of this semigroup in $L^p(\Omega)$ is given by
\[A_p:\{f\in L^p(0,1);xf'(x)\in L^p(0,1)\}\rightarrow L^p(\Omega), A_pf(x)=-xf'(x)+h(x)f(x),\]
where $f'$ denotes again the distributional derivative of $f$ (see e.g.\ \cite[Theorem 15]{ArKaMa13}). 

If $x\mapsto\frac{\Im h(x)}{x}\in L^1(0,\beta)$ for all $\beta\in (0,1)$ or if $x\mapsto\frac{\Im h(x)}{x}\in L^1(\beta,1)$ for all $\beta\in (0,1)$, by Theorem \ref{simplified chaos} this $C_0$-semigroup is chaotic on $L^p(\Omega)$ precisely when for some $x\in(0,1)$
\[\int_{(0,1)}\exp(-p\int_x^w\frac{\Re h(y)}{-y}dy)d\lambda(w)<\infty.\]
Because of
\[\exp(p\int_x^w\frac{\Re h(y)}{y}dy)=\Big(\frac{w}{x}\Big)^{p\Re h(0)}\exp(p\int_x^w\frac{\Re h(y)-\Re h(0)}{y}dy)\]
this generalizes a result of Dawidowicz and Poskrobko \cite{DaPo} who showed that in case of a real valued $h\in C[0,1]$ for which $x\mapsto\frac{h(x)-h(0)}{x}\in L^1(0,1)$ the above semigroup is chaotic on $L^p(0,1)$ if and only if $h(0)>-1/p$.

d) Consider $\Omega=(0,1)$ and $F(x)=-x^3\sin(\frac{1}{x})$. Because we have $\lim_{x\rightarrow 0}F(x)= 0$ and $\lim_{x\rightarrow 1}F(x)\leq 0$ it follows that $\Omega$ is forward invariant under $F$ and since $F'$ is bounded $\rho=1$ is $p$-admissible for $F$ and $h=0$ for every $1\leq p<\infty$. Thus, $T_F$ is a well-defined $C_0$-semigroup on $L^p(0,1)$. By \cite[Theorem 15]{ArKaMa13} its generator is
\begin{eqnarray*}
&&A_p:\{f\in L^p(0,1);\,-x^3\sin(\frac{1}{x})f'(x)\in L^p(0,1)\}\rightarrow L^p(0,1),\\
&&A_pf(x)=-x^2\sin(\frac{1}{x})f'(x)
\end{eqnarray*}
where $f'$ denotes the distributional derivative of $f$. By Remark \ref{remark} it follows that this $C_0$-semigroup is chaotic on $L^p(0,1)$ for every $1\leq p<\infty$. 
\end{rm}
\end{example}

\section{Weighted composition $C_0$-semigroups on Sobolev spaces}

For a bounded interval $(a,b)$, let $F\in C^1[a,b]$ with $F(a)=0$ be such that $(a,b)$ is forward invariant under $F$, and let $h\in W^{1,\infty}[a,b]$ be such that
\begin{itemize}
	\item[1)] $\forall\,x\in\{F=0\}:\,h(x)=h(a)\in\R$,
	\item[2)] the function $[a,b]\rightarrow\R, y\mapsto\frac{h(y)-h(a)}{F(y)}$ belongs to $L^\infty[a,b]$.
\end{itemize}
In \cite{ArKaMa13} it is shown that under the above hypothesis the operator
\[A_p:\{f\in W^{1,p}[a,b];\, Ff''\in L^p[a,b]\}\rightarrow W^{1,p}[a,b], A_pf=Ff'+hf,\]
where the derivatives are taken in the distributional sense, is the generator of a $C_0$-semigroup $S_{F,h}$ on $W^{1,p}[a,b]\, (1\leq p<\infty)$ which is given by
\[\forall\,t\geq 0, f\in W^{1,p}[a,b]:\,S(t)f(x)=h_t(x)f(\varphi(t,x)).\]
Moreover, it is shown in \cite{ArKaMa13} that this $C_0$-semigroup $S_{F,h}$ is never hypercyclic on $W^{1,p}[a,b]$. In particular, $S_{F,h}$ cannot be chaotic on $W^{1,p}[a,b]$. 

Because of $F(a)=0$, the closed subspace
\[W^{1,p}_*[a,b]:=\{f\in W^{1,p}[a,b];\,f(a)=0\}\]
of $W^{1,p}[a,b]$ is invariant under $S_{F,h}$ such that the restriction of $S_{F,h}$ to $W^{1,p}_*[a,b]$ defines a
$C_0$-semigroup on $W^{1,p}_*[a,b]$ which we denote again by $S_{F,h}$. Its generator is given by
\[A_{p,*}:\{f\in W^{1,p}_*[a,b];\, Ff''\in L^p[a,b]\}\rightarrow W^{1,p}[a,b], A_{p,*}f=Ff'+hf,\]
see \cite{ArKaMa13}. Using Theorem \ref{simplified chaos} we derive the following characterization of chaos for $S_{F,h}$ on $W^{1,p}_*[a,b]$.

\begin{theo}\label{simplified in Sobolev}
Let $(a,b)$ be a bounded interval, $F\in C^1[a,b]$ with $F(a)=0$ such that $(a,b)$ is forward invariant under $F$.
Moreover, let $h\in W^{1,\infty}[a,b]$ be such that
\begin{itemize}
	\item[1)] $\forall\, x\in\{F=0\}:\,h(x)=h(a)\in\R$,
	\item[2)] the function $[a,b]\rightarrow\C, y\mapsto\frac{h(y)-h(a)}{F(y)}$ belongs to $L^\infty[a,b]$.
	\end{itemize}
	Then, for the $C_0$-semigroup $S_{F,h}$ on $W^{1,p}_*[a,b]$ the following are equivalent.
	\begin{itemize}
		\item[i)] $S_{F,h}$ is chaotic.
		\item[ii)] $\lambda(\{F=0\})=0$ and for every connected component $C$ of $(a,b)\backslash\{F=0\}$
		\[\int_C\exp(-p\int_x^w\frac{F'(y)+ h(a)}{F(y)}dy)d\lambda(w)<\infty\]
		for some/all $x\in C$.
	\end{itemize}
\end{theo}

\begin{proof}
Observe that by the boundedness of $F'$ on $[a,b]$ $\rho=1$ is $p$-admissible for $F$ and $F'+h(a)$ for any $1\leq p<\infty$. Under the above hypothesis 1) and 2) it is shown in \cite[Theorem 20 and Proposition 24]{ArKaMa13} that
the $C_0$-semigroups $S_{F,h}$ on $W^{1,p}_*[a,b]$ and $T_{F,F'+h(a)}$ on $L^p[a,b]$ are conjugate, i.e.\ there is a
homeomorphism $\Phi:L^p[a,b]\rightarrow W^{1,p}_*[a,b]$ such that $S_{F,h}(t)\circ\Phi=\Phi\circ T_{F,F'+h(a)}(t)$ for
every $t\geq 0$. By the so-called Comparison Principle (see e.g.\ \cite[Proposition 7.7]{GEPe11}) it follows that
$S_{F,h}$ is chaotic on $W^{1,p}_*[a,b]$ if and only if $T_{F,F'+h(a)}$ is chaotic on $L^p[a,b]$. Thus, an application of Theorem \ref{simplified chaos} proves the theorem.
\end{proof}

\begin{example}
\begin{rm}
a) We consider $(a,b)=(0,1)$ and $F(x)=-x$. Then, $(0,1)$ is forward invariant under $F$. For every $h\in W^{1,\infty}[0,1]$ with $h(0)\in\R$ and
\[[0,1]\rightarrow\C,y\mapsto\frac{h(y)-h(0)}{y}\in L^\infty[0,1]\]
the operator
\[A:\{f\in W^{1,p}_*[a,b];\, xf''(x)\in L^p[a,b]\}\rightarrow W^{1,p}[a,b], Af(x)=-xf'(x)+h(x)f(x),\]
generates a $C_0$-semigroup on $W^{1,p}_*[0,1], 1\leq p<\infty$. By Theorem \ref{simplified in Sobolev} this semigroup is chaotic on $W^{1,p}_*[0,1]$ if and only if for some $x\in (0,1]$
\[\int_{[0,1]}\Big(\frac{w}{x}\Big)^{p(h(0)-1)}d\lambda(w)=\int_{[0,1]}\exp(-p\int_x^w\frac{-1+ h(0)}{-y}dy)d\lambda(w)<\infty\]
which holds precisely when $p(h(0)-1)>-1$, i.e.\ when $h(0)>1-\frac{1}{p}$ (see also \cite[Theorem 27]{ArKaMa13}).

b) Let again $(a,b)=(0,1)$. We consider $F(x)=-x(1-x)$ so that $(0,1)$ is forward invariant under $F$. For each $h\in W^{1,\infty}[0,1]$ with $h(0)=h(1)\in\R$ and
\[[0,1]\rightarrow\C,y\mapsto\frac{h(y)-h(0)}{y(1-y)}\in L^\infty[0,1]\]
the operator
\begin{eqnarray*}
&&A:\{f\in W^{1,p}_*[a,b];\, x(1-x)f''(x)\in L^p[a,b]\}\rightarrow W^{1,p}[a,b],\\
&&Af(x)=-x(1-x)f'(x)+h(x)f(x),
\end{eqnarray*}
generates a $C_0$-semigroup on $W^{1,p}_*[0,1], 1\leq p<\infty$. Since for any $x\in (0,1)$ the function
\[w\mapsto\exp(-p\int_x^w\frac{F'(y)-h(0)}{F(y)}dy)=w^{-p(1+h(0))}(1-w)^{-p(1-h(0))}(1-x)^{p(1-h(0))}x^{p(1+h(0))}\]
does not belongs to $L^1(0,1)$ for any value of $h(0)$ it follows from Theorem \ref{simplified in Sobolev} that this semigroup is not chaotic.
\end{rm}
\end{example}

\begin{small}
{\sc Technische Universit\"at Chemnitz,
Fakult\"at f\"ur Mathematik, 09107 Chemnitz, Germany}

{\it E-mail address: thomas.kalmes@mathematik.tu-chemnitz.de}
\end{small}


\begin{thebibliography}{99}
	\bibitem{AlBaMaPe13} A.\ Albanese, X.\ Barrachina, E.\ Mangino, A.\ Peris, Distributional chaos for strongly continuous semigroups of operators, {\it Commun.\ Pure Appl.\ Analysis} 12: 2069--2082, 2013.
	\bibitem{Amann} H.\ Amann, {\it Ordinary differential equations. An introduction to nonlinear
  analysis}. Volume 13 of De Gruyter Studies in Mathematics. Walter De Gruyter, Berlin-New York, 1990.
	\bibitem{ArKaMa13} J.\ Aroza, T.\ Kalmes, E.\ Mangino, Chaotic $C_0$-semigroups induced by semiflows in Lebesgue and
	Sobolev spaces, {\it J.\ Math.\ Anal.\ Appl.}, 412:77--98, 2014.
	\bibitem{ArPe12} J.\ Aroza, A.\ Peris, Chaotic behaviour and birth-and-death models with proliferation, {\it J.\ Difference Equ.\ Appl.} 18:647--655,2012.
	\bibitem{BaMo11} J.\ Banasiak, M.\ Moszynski, Dynamics of 
birth-and-death processes with
proliferation-stability and chaos, {\it Discrete Contin.\ Dyn.\ Syst.} 29: 67--79, 2011.
	\bibitem{DaPo} A.L.\ Dawidowicz, A.\ Poskrobko, On chaotic and stable behaviour of the von Foerster-Lasota equation in some Orlicz spaces, {\it Proc.\ Est.\ Acad.\ Sci.} 57(2):61--69, 2008.
	\bibitem{deLaEm01} R.\ deLaubenfels, H.\ Emamirad, Chaos for functions of discrete and continuous weighted shift operators, {\it Ergodic Theory Dynam.\ Systems}, 21(5):1411--1427, 2001.
	\bibitem{EmGoGo12} H.\ Emamirad, G.R.\ Goldstein, J.A.\ Goldstein, Chaotic solution for the Black-
Scholes equation, {\it Proc.\ Amer.\ Math.\ Soc.} 140:2043--2052, 2012.
	\bibitem{GEPe11} K.G.\ Grosse-Erdmann, A.\ Peris Manguillot, {\it Linear chaos}, Universitext. Springer, London, 2011.
	\bibitem{Ka07} T.\ Kalmes, Hypercyclic, mixing, and chaotic {$C\sb 0$}-semigroups induced by semiflows, {\it Ergodic Theory Dynam.\ Systems}, 27(5):1599--1631, 2007.
	\bibitem{Ka09} T.\ Kalmes, Hypercyclic $C_0$-semigroups and evolution families generated by first order differential
	operators, {\it Proc.\ Amer.\ Math.\ Soc.} 137(11):3833--3848, 2009.
	\bibitem{MaYaTa03} M.\ Matsui, M.\ Yamada, F.\ Takeo, Supercyclic and chaotic translation semigroups, {\it Proc.\ Amer.\ Math.\ Soc.} 131(11):3535--3546, 2003.
	\bibitem{MaYaTa04} M.\ Matsui, M.\ Yamada, F.\ Takeo, Erratum to 'Supercyclic and chaotic translation semigroups', {\it Proc.\ Amer.\ Math.\ Soc.} 132(12):3751--3752, 2004.
	\bibitem{Ru12} R.\ Rudnicki, Chaoticity and invariant measures for a cell population model, {\it J.\ Math.\ Anal.\ Appl.}, 393(1):151--165, 2012.

\end{thebibliography}
\end{document}